\documentclass[11pt]{amsart}
\usepackage{amssymb}
\usepackage{amsmath}
\usepackage{amsthm}
\usepackage{latexsym}
\usepackage{amsfonts}
\usepackage{graphicx}
\usepackage{graphics}
\usepackage[T1]{pbsi}

\newtheorem{thm}{Theorem}[section]   % Αρίθμηση kατά Θεώρημα, Λήμμα κ.λπ.)
\newtheorem{lem}{Lemma}%[thm]
\newtheorem{cor}{Corollary}%[thm]
\newtheorem{Def}{Definition}%[thm]

\theoremstyle{definition}
\newtheorem{rem}{Remark}%[thm]
\newtheorem*{Proof}{Proof}

\newcommand{\dis}{\displaystyle}
\newcommand{\ra}{\;\rightarrow\;}

\newcommand{\bi}{\beta}

\newcommand{\de}{\delta }

\newcommand{\De} {{\varDelta}}

\newcommand{\la}{\lambda }
\newcommand{\mi}{\mu }

\newcommand{\oo}{\omega}

\newcommand{\R}{\mathbb{R}}
\newcommand{\Z}{\mathbb{Z}}
\newcommand{\N}{\mathbb{N}}

\newcommand{\intl}{\int\limits}

\newcommand{\ct}{{\mathcal{T}}}

\newcommand{\cm}{{\mathcal{M}}}

\newcommand{\ld}{\ldots}

\newcommand{\sm}{\smallsetminus}

 \newcommand{\loc}{\mbox{\footnotesize loc}}
\newcommand{\hs}{\hfill$\square$}
    \newcommand{\vs}{\vspace*{0.2cm} \\}

\begin{document}

\title[Dyadic Maximal Operators]{The Bellman function of the dyadic maximal operator related to Kolmogorov's inequality}
\author{Eleftherios N. Nikolidakis}\thanks{\hspace*{-0.5cm}{\em{Keywords}}: Bellman, dyadic, maximal}%\\
%{\small Department of Mathematics and Statistics},\vspace*{-0.15cm}\\
%{\small University of Cyprus 20537, CY 1678, Nicosia, Cyprus}}
%\footnotetext{\hspace{-0.5cm}} \footnotetext{\hspace{-0.5cm}E-mail
%address: lefteris@math.uoc.gr}
\date{}
\maketitle
\noindent
{\bf Abstract.} We precisely compute the Bellman function of two variables of the dyadic maximal operator in relation to Kolmogorov's inequality. In this way we give an alternative proof of the results in \cite{4}.\ Additionally we characterize the sequences of functions that are extremal for this Bellman function.\ The proof for this is based on that is given in this paper for the Bellman function we are interested in.
\section{Introduction}\label{sec1}
\noindent

The dyadic maximal operator is defined on $\R^n$ by
\begin{eqnarray}
\hspace*{1.5cm}\cm_d\phi(x)=\sup\bigg\{\frac{1}{|Q|}\int_Q|\phi(u)|du:x\in Q,\;Q\subseteq\R^n\;\text{is a dyadic cube}\bigg\},  \label{eq1.1}
\end{eqnarray}

for every $\phi\in L^1_{\loc}(\R^n)$, where the dyadic cubes are those formed by the grids $2^{-N}\Z^n$, for $N=0,1,2,\ld\;,$
and $|\cdot|$ is the Lesbesgue measure on $R^n$.
As it is well known it satisfies the following weak type (1,1) inequality
\begin{eqnarray}
\hspace*{1cm}|\{x\in\R^n:\cm_d\phi(x)>\la\}|\le\frac{1}{\la}\int_{\{\cm_d\phi>\la\}}
|\phi(u)|du  \label{eq1.2}
\end{eqnarray}
for every $\phi\in L^1(\R^n)$ and every $\la>0$, from which it is easy to get the following $L^p$-inequality:
\begin{eqnarray}
\|\cm_d\phi\|_p\le\frac{p}{p-1}\|\phi\|_p  \label{eq1.3}
\end{eqnarray}
for every $\phi\in L^p(\R^n)$, $p>1$.

It is easy to see that (\ref{eq1.2}) is best possible while it has also been proved that (\ref{eq1.3}) is sharp (see \cite{1} and \cite{2} for general martingales and \cite{16} for dyadic ones).

Our aim is to study the dyadic maximal operator and this can be done by finding refinements of the above inequalities.
Concerning (\ref{eq1.2}) refinements have been studied in \cite{8} and \cite{9}. For the study of (\ref{eq1.3}) the following function has been precisely computed in \cite{3}:
\begin{align}
B^Q_p(f,F)=\sup&\left\{\frac{1}{|Q|}\int_Q(\cm_d\phi)^p:\phi\ge0,\frac{1}{|Q|}
\int_Q\phi(u)du=f,\right. \nonumber \\
&\ \ \left.\frac{1}{|Q|}\int_Q\phi^p(u)du=F\right\}.  \label{eq1.4}
\end{align}
where $Q$ is a fixed dyadic cube on $\R^n$ and $f,F$ are variables satisfying: $0<f^p\le F$. It's exact value has been found to be equal to
\[
B^Q_p(f,F)=F\oo_p(f^p/F)^p
\]
where $\oo_p:[0,1]\ra\Big[1,\dfrac{p}{p-1}\Big]$ is the inverse function of $H_p$, which is given by $H_p(z)=-(p-1)z^p+pz^{p-1}$, for $z\in\Big[1,\dfrac{p}{p-1}\Big]$.\ After completing the case $p>1$ it is interesting to search for the case where $p=q<1$ and as it is well known it is connected with the following known as Kolmogorov's inequality
\begin{eqnarray}
\int_E(\cm_d\phi(u))^qdu\le\frac{1}{1-q}|E|^{1-q}\bigg(\int_{\R^n}|\phi|\bigg)^q \label{eq1.5}
\end{eqnarray}
for every $q\in(0,1)$, $\phi\in L^1(\R^n)$ and $E$ measurable subset of $\R^n$ with finite measure.

This inequality connects the $L^q$ norm of $\cm_d\phi$ upon subsets of $\R^n$ of finite measure with the $L^1$-norm of $\phi$ and the measure of the set. It was studied extensively in \cite{4} and it is proved there that it is sharp.\ More precisely a stronger result than it's  sharpness is proved, namely the exact evaluation of the following function of four variables $f,h,L,k$:
\begin{align}
&B_q(f,h,L,k)=\sup\left\{\frac{1}{|Q|}\int_E(\cm_d\phi)^q:\phi\ge0,\frac{1}{|Q|}
\int_Q\phi=f,\right. \nonumber \\
&\left.\frac{1}{|Q|}\int_Q\phi^q=h,\;\sup_{Q'\supseteq Q}\bigg(\frac{1}{|Q'|}\int_{Q'}\phi\bigg)=L,\; E\subseteq Q\;\text{measurable with $|E|=k$}\right\}  \label{eq1.6}
\end{align}
where $Q$ is a fixed dyadic cube, $Q'$ runs over all the dyadic cubes containing $Q$, $\phi\in  L^1(Q)$, $0<k\le|Q|$ and $f,h,L$ satisfy $0<f\le L$, $h\le f^q$.

It turns out that $(1.6)$ is independent of $Q$ so we can consider $Q=[0,1]^n$. More generally we consider a non-atomic probability measure space $(X,\mi)$ equipped with a tree structure $\ct$, which plays the role of the dyadic sets in our situation (see definition in Section \ref{sec2}).

Then the dyadic maximal operator $\cm_\ct$ is defined by:
\begin{eqnarray}
\cm_\ct\phi(x)=\sup\bigg\{\frac{1}{\mi(I)}\int_I|\phi|d\mi:\;x\in I\in\ct\bigg\} \label{eq1.5}
\end{eqnarray}
for every $\phi\in L^1(X,\mi)$.

It is not difficult to see that (\ref{eq1.2}) and (\ref{eq1.3}) remain true and sharp even in this more general setting.

We define now
\[
B'_q(f,h,L,k)=\sup\bigg\{\int_E\max(\cm_\ct\phi,L)^qd\mi,\;\phi\ge0,\;\phi\in L^1(X,\mi),
\]
\begin{eqnarray}
\int_X\phi d\mi=f,\;\int_X\phi^qd\mi=h,\;E\subseteq X\;\text{measurable with}\;
\mi(E)=k\bigg\}, \label{eq1.6}
\end{eqnarray}
where $L,f,h,k$ satisfy $L\ge f>0$, $0<h\le f^q$, $0<k\le1$.

Then it is true that $B'_q=B_q$ according to arguments given in \cite{3}.

The precise value of $B'_q$ has been found by working the respective Bellman function of two variables which is defined by,
\begin{eqnarray}
B_q(f,h)=\sup\bigg\{\int_X(\cm_\ct\phi)^qd\mi:\;\phi\ge0,\;\int_X\phi d\mi=f,\;\int_X\phi^qd\mi=h\bigg\}  \label{eq1.7}
\end{eqnarray}
with $0<h\le f^q$.

Several calculus arguments and the use of the value of (\ref{eq1.7}) in certain subsets of $X$ gives (\ref{eq1.6}) as is done in \cite{4}. We are thus interested in (\ref{eq1.7}). The result is the following:\vs
\noindent
{\bf Theorem 1.} {\em It is true that:
\begin{eqnarray}
B_q(f,h)=h\oo_q(f^q/h), \ \ \text{where} \ \ \oo_q:[1,+\infty)\ra[1,+\infty)  \label{eq1.8}
\end{eqnarray}
is defined by $\oo_q(z)=[H^{-1}_q(z)]^q$ where}
\[
H_q(z)=(1-q)z^q+qz^{q-1}, \ \ z\ge1.
\]

Our first aim in this paper is to give an alternative proof of Theorem 1.

Our second aim is to characterize the extremal sequences of functions concerning (\ref{eq1.7}).\ More precisely we will prove the following.\vspace*{0.2cm}\\
\noindent
{\bf Theorem 2.} {\em Let $\phi_n:(X,\mi)\ra\R^+$ be such that $\int\limits_X\phi_ndh=f$ and $\int\limits_X\phi^q_nd\mi=h$, for any $n\in\N$.\ Then the following are equivalent
\begin{enumerate}
\item[i)] $\dis\lim_n\int\limits_X(\cm_T\phi_n)^qd\mi=h\oo_q(f^q/h)$.
\item[ii)] $\dis\lim_n\int\limits_X|\cm_\ct\phi_n-c\phi_n|^qd\mi=0$, where $c=\oo_q(f^q/h)^{1/q}$.
\end{enumerate}

That is $\phi_n$ behaves approximately in $L^q$ like eigenfunction of $\cm_\ct$ for the eigenvalue $c$.}

We also remark that there are several problems in Harmonic Analysis were Bellman functions arise. Such problems (including the dyadic Carleson imbedding theorem and weighted inequalities) are described in \cite{12} (see also \cite{5}, \cite{6}) and also connections to Stochastic Optimal Control are provided, from which it follows that the corresponding Bellman functions satisfy certain nonlinear second-order PDEs. The exact evaluation of a Bellman function is a difficult task which is connected with the deeper structure of the corresponding Harmonic Analysis problem. Until now several Bellman functions have been computed (see \cite{1}, \cite{2}, \cite{3}, \cite{4}, \cite{5}, \cite{6}, \cite{7}, \cite{12}, \cite{13}, \cite{14}, \cite{15},). The exact evaluation of (\ref{eq1.7}) for $q>1$ has been also given in \cite{11} by L. Slavin, A. Stokolos and V. Vasyunin which linked the computation of it to solving certain PDEs of the Monge-Amp\`{e}re type and in this way they obtained an alternative proof of the results in \cite{3} for the Bellman functions related to the dyadic maximal operator.

The paper is organized as follows.

In Section \ref{sec2} we give some preliminary results and facts needed for use in the subsequent sections.
In Section \ref{sec3} we give a proof that the right side of (\ref{eq1.8}) is an upper bound for $\intl_X(\cm_\ct\phi)^qd\mi$.
In Section \ref{sec4} we give the sharpness of the result just mentioned.
In Section \ref{sec5} we prove Theorem 2.\
At last in Section \ref{sec6} we discuss further properties of certain extremal sequences for the Bellman function (1.9).

We need also to say that analogous results for the case $q>1$ are treated in \cite{7} but for the Bellman function of three variables. More precisely in \cite{7} it is proved a generalization of a symmetrization principle which is presented in \cite{10}.

\section{Preliminaries}\label{sec2}
\noindent

Let $(X,\mi)$ be a non-atomic probability measure space. We give the following.
\begin{Def}\label{Def2.1}
A set $\ct$ of measurable subsets of $X$ will be called a tree if it satisfies the following conditions
\begin{enumerate}
\item[i)] $X\in\ct$ and for every $I\in\ct$ we have that $\mi(I)>0$.
\item[ii)] For every $I\in\ct$ there corresponds a finite or countable subset $C(I)\subseteq\ct$ containing at least two elements such that
\begin{itemize}
\item[(a)] the elements of $C(I)$ are pairwise disjoint subsets of $I$
\item[(b)] $I=\cup C(I)$.
\end{itemize}
\item[iii)] $\ct=\bigcup\limits_{m\ge0}\ct_{(m)}$ where $\ct_{(0)}=\{X\}$ and $\ct_{(m+1)}=\bigcup\limits_{I\in\ct_{(m)}}C(I)$.
\item[iv)] We have that $\dis\lim_{m\ra\infty}\dis\sup_{I\in\ct_{(m)}}\mi(I)=0$.  \hs
\end{enumerate}
\end{Def}

Examples of trees are given in \cite{3}. The most known is the one given by the family of dyadic subcubes of $[0,1]^n$.

The following has been proved in \cite{10}.
\begin{thm}\label{thm2.1}
For any $g:(0,1]\ra\R^+$ non-increasing, every increasing function $G_1$ defined on $[0,+\infty)$ with non-negative values and every $k\in(0,1]$ the following holds:
\begin{align*}
&\sup\bigg\{\int_KG_1(\cm_\ct\phi)d\mi:\phi\ge0,\phi^\ast=g,\;K\;\text{measurable subset of $X$ with} \;\mi(K)=k\bigg\}\\
&=\int^k_0G_1\bigg(\frac{1}{t}\int^t_0g\bigg)dt.
\end{align*}
\hs
\end{thm}

Here by $\phi^\ast$ we mean the decreasing rearrangement of $\phi$ defined by
\[
\phi^\ast(t)=\sup_{e\subset X\atop |e|=t}\inf_{x\in e}|\phi(x)|, \ \ 0<t\le1.
\]
which is a function equimeasurable to $\phi$, non-increasing and left continuous.

We remind that given a tree on $(X,\mi)$ we define the associated dyadic maximal operator as follows
\[
\cm_\ct\phi(x)=\sup\bigg\{\frac{1}{\mi(I)}\int_I|\phi|d\mi:\;x\in I\in\ct\bigg\}
\]
for every $\phi\in L^1(X,\mi)$.
\section{The Bellman function}\label{sec3}
\noindent

We are now able to prove the following
\setcounter{lem}{0}
\begin{lem}\label{lem3.1}
For every $q$ such that $0<q<1$ and every $f,h$ such that $0<h\le f^q$ we have that
\setcounter{equation}{0}
\begin{eqnarray}
\int_X(\cm_\ct\phi)^qd\mi\le h\oo_q\bigg(\frac{f^q}{h}\bigg)  \label{eq3.1}
\end{eqnarray}
for any $\phi\in L^1(X,\mi)$ such that $\phi\ge0$, $\intl_X\phi d\mi=f$, $\intl_X\phi^qd\mi=h$.
\end{lem}
\begin{Proof}
We set $I=\intl_X(\cm_\ct\phi)^qd\mi$. Then
\begin{align*}
I&=\int^{+\infty}_{\la=0}q\la^{q-1}\mi(\{x\in X:\cm_\ct\phi(x)>\la\})d\la\\
&=\int^f_{\la=0}+\int^{+\infty}_{\la=f}q\la^{q-1}\mi(\{x\in X:\;\cm_\ct\phi(x)>\la\})d\la=II+III, \ \ \text{where}
\end{align*}
\[
II=\int^f_{\la=0}q\la^{q-1}d\la=f^q, \ \ \text{and}
\]
\[
II=\int^{+\infty}_{\la=f}q\la^{q-1}\mi(\{x\in X:\;\cm_\ct\phi(x)>\la\})d\la.
\]
Now because of the weak type inequality (\ref{eq1.2}) for $\cm_\ct$ we have that
\begin{align*}
III&\le\int^{+\infty}_{\la=f}q\la^{q-1}\frac{1}{\la}\bigg(\int_{\{\cm_\ct\phi>\la\}}\phi
d\mi\bigg)d\la \\
&=\int^{+\infty}_{\la=f}q\la^{q-2}\bigg(\int_{\{\cm_\ct\phi>\la\}}\phi d\mi\bigg)d\la \ \ \text{(by Fubini's theorem)} \\
&=\int_X\phi(x)\frac{q}{q-1}\big[\la^{q-1}\big]^{\cm_\ct\phi(x)}_{\la= f}d\mi(x) \\
&=\frac{q}{1-q}f^q-\frac{q}{1-q}\int_X\phi(x)[\cm_\ct\phi(x)]^{q-1}d\mi(x).
\end{align*}

Thus we have that
\begin{eqnarray}
I=\int_X(\cm_\ct\phi)^qd\mi\le\frac{1}{1-q}f^q-\frac{q}{1-q}IV,  \label{eq3.2}
\end{eqnarray}
where
\[
IV=\int_X\phi(\cm_\ct\phi)^{q-1}d\mi.
\]

On the other hand we know from Holder's inequality that the following is true
\begin{eqnarray}
\int_X(\phi_1\phi_2)^qd\mi\le\bigg(\int_X\phi_1d\mi\bigg)^q\cdot\bigg(\int_X\phi^{q/(1-q)}_2d\mi
\bigg)^{1-q}  \label{eq3.3}
\end{eqnarray}
for any $\phi_1,\phi_2$ such that $\phi_1\in L^1(X,\mi)$, $\phi_2\in L^{q/(1-q)}(X,\mi)$, where $q\in(0,1)$.
We set $\phi_1=\phi(\cm_\ct\phi)^{q-1}$ and $\phi_2=(\cm_\ct\phi)^{1-q}$ in (\ref{eq3.3}), and we conclude that
\begin{align}
h&=\int_X\phi^qd\mi\le\bigg[\int_X\phi(\cm_\ct\phi)^{q-1}d\mi\bigg]^q\cdot
\bigg[\int_X(\cm_\ct\phi)^qd\mi\bigg]^{1-q} \nonumber\\
&=IV^q\cdot I^{1-q}\Rightarrow IV\ge h^{\frac{1}{q}}I^{1-\frac{1}{q}},  \label{eq3.4}
\end{align}
Thus (\ref{eq3.2}) in view of (\ref{eq3.4}) gives
\begin{align*}
&I\le\frac{1}{1-q}f^q-\frac{q}{1-q}h^{\frac{1}{q}}I^{1-\frac{1}{q}}\Rightarrow(1-q)
\frac{I}{h}\le\frac{f^q}{h}-q\bigg(\frac{I}{h}\bigg)^{1-\frac{1}{q}}\Rightarrow\bigg(u=\frac{I}{h}\bigg)\\
&qu^{1-\frac{1}{q}}+(1-q)u\le\frac{f^q}{h}\Rightarrow u\le\oo_q\bigg(\frac{f^q}{h}\bigg)
\end{align*}
in case where $u\ge1$, while $u\le 1\le\oo_q\Big(\dfrac{f^q}{h}\Big)$ in case where $u<1$, because of the definition of $\oo_q(z)$, $z\ge1$.

Thus we have that
\[
I=\int_X(\cm_\ct\phi)^qd\mi\le h\oo_q\bigg(\frac{f^q}{h}\bigg).
\]
and in this way we derive the proof of our Lemma. \hs
\end{Proof}

We will also need the following
\begin{lem}\label{lem3.2}
For any $g:(0,1]\ra\R^+$ non-increasing, with $\intl^1_0g(u)du=f$, and any $q$ such that $0<q<1$, the following equality holds:
\[
\int^1_0\bigg(\frac{1}{t}\int^t_0g\bigg)^qdt=\frac{1}{1-q}f^q-\frac{q}{1-q}\int^1_0
g(t)\bigg(\frac{1}{t}\int^t_0g\bigg)^{q-1}dt.
\]
\end{lem}
\begin{Proof}
The proof is similar to that of Lemma \ref{lem3.1}.

We set
\[
I=\int^1_0\bigg(\frac{1}{t}\int^t_0g\bigg)^qdt=f^q+\int^{+\infty}_{\la=f}q\la^{q-1}
\bigg|\bigg\{t\in(0,1]:\frac{1}{t}\int^t_0g>\la\bigg\}\bigg|d\la.
\]
We consider for every $\la>f$ the unique real number on $(0,1]$, $\bi(\la)$ such that $\dfrac{1}{\bi(\la)}\intl^{\bi(\la)}_0g=\la$ (without loss of generality $g(0+)=+\infty$, the finite case is treated similarly).

Then because of the monotonicity of $g$, for any $\la>f$
\[
\bigg\{t\in(0,1]:\frac{1}{t}\int^t_0g>\la\bigg\}=(0,\bi(\la)), \ \ \text{so that}
\]
\begin{align*}
I&=f^q+\int^{+\infty}_{\la=f}q\la^{q-1}\bi(\la)d\la
=f^q+\int^{+\infty}_{\la=f}q\la^{q-1}\frac{1}{\la}\bigg(\int^{\bi(\la)}_0
g(u)du\bigg)d\la \\
&=f^q+\int^{+\infty}_{\la=f}q\la^{q-2}\bigg(\int_{\big\{t:\dfrac{1}{t}\intl^t_0g>\la\big\}}g(u)du
\bigg)d\la=f^q+\int^1_0g(t)\frac{q}{q-1}\big[\la^{q-1}\big]^{\frac{1}{t}\intl^t_0g}_{\la=f}
d\la \\
&=\frac{1}{q-1}f^q-\frac{q}{1-q}\int^1_0g(t)\bigg(\frac{1}{t}\int^t_0g\bigg)^{q-1}dt,
\end{align*}
and Lemma \ref{lem3.2} is proved. \hs
\end{Proof}
\section{Sharpness of Lemma \ref{lem3.1}}\label{sec4}
\noindent

In the determination of the upper bound of $\intl_X(\cm_\ct\phi)^qd\mi$ in Lemma \ref{lem3.1} there are exactly two steps where inequalities are used.

The first is before we reach to the following inequality
\setcounter{equation}{0}
\begin{eqnarray}
\int_X(\cm_\ct\phi)^qd\mi\le\frac{1}{1-q}f^q-\frac{q}{1-q}\int_X\phi(\cm_\ct\phi)^{q-1}d\mi,
\label{eq4.1}
\end{eqnarray}
while by Lemma \ref{lem3.2} we have equality in the respective inequality for the Hardy operator, this is
\begin{eqnarray}
\int^1_0\bigg(\frac{1}{t}\int^t_0g\bigg)^qdt=\frac{1}{1-q}f^q-\frac{q}{1-q}
\int^1_0g(t)\bigg(\frac{1}{t}\int^t_0g\bigg)^{q-1}dt. \label{eq4.2}
\end{eqnarray}

We now use Theorem \ref{thm2.1} of Section \ref{sec2}, which states that
\begin{eqnarray}
\int_X(\cm_\ct\phi)^qd\mi\le\int^1_0\bigg(\frac{1}{t}\int^t_0g\bigg)^qdt,  \label{eq4.3}
\end{eqnarray}
with $\phi^\ast=g$, which is sharp when one considers all $\phi$ such that $\phi^\ast=g$.

Thus we observe that if we fix $g$, and leave $\phi$ run across all the rearrangements of $g$ we attain equality in the first inequality which we meet in Lemma \ref{lem3.1} and this is exactly  Lemma \ref{lem3.2}. As for the second step where an inequality was used in the proof of Lemma \ref{lem3.1} we need to mention the following.

Because of (\ref{eq3.3}) we have that
\begin{eqnarray}
\int^1_0g(t)\bigg(\frac{1}{t}\int^t_0g\bigg)^{q-1}dt\ge h^{\frac{1}{q}}\bigg[\int^1_0\bigg(\frac{1}{t}\int^t_0g\bigg)^qdt\bigg]^{1-\frac{1}{q}}. \label{eq4.4}
\end{eqnarray}
Now (\ref{eq4.2})-(\ref{eq4.4}) give for any $\phi:\phi^\ast=g$
\begin{eqnarray}
\hspace*{1cm}\int_X(\cm_\ct\phi)^qd\mi\le\int^1_0\bigg(\frac{1}{t}\int^t_0g\bigg)^qdt=I_g \le\frac{1}{1-q}f^q-\frac{q}{1-q}h^{\frac{1}{q}}I^{1-\frac{1}{q}}_g.  \label{eq4.5}
\end{eqnarray}
(\ref{eq4.5}) now gives as in the proof of Lemma \ref{lem3.1} that $I_g\le h\oo_q(f^q/h)$.

So if we want to attain equality in the last relation we need equality on (\ref{eq4.4}) that is we must find a non-increasing function $g$ for which we have equality in a Holder type inequality. Thus this function should satisfy the following
\[
\frac{1}{t}\int^t_0g=cg(t), \ \ t\in(0,1]
\]
for some constant $c$. If additionally $\int^1_0g=f$, $\int^1_0g^q=h$ and $c=\oo_q\Big(\dfrac{f^q}{h}\Big)^{1/q}$, then in view of the above discussion we will have that
\[
\sup_{\phi^\ast=g}\int_X(\cm_\ct\phi)^qd\mi=\oo_q(f^q/h)\cdot\int^1_0g^q=
h\oo_q\bigg(\frac{f^q}{h}\bigg), \ \ \text{for that $g$}.
\]
Thus the following will give the sharpness of Lemma \ref{lem3.1}.
\setcounter{lem}{0}
\begin{lem}\label{lem4.1}
For any $f,h$: $0<h\le f^q$ there exists $g:(0,1]\ra\R^+$ non-increasing and continuous such that the following hold $\int^1_0g(u)du=f$, $\int^1_0g^q(u)du=h$ and
\[
\frac{1}{t}\int^t_0g(u)du=\oo_q\bigg(\frac{f^q}{h}\bigg)^{1/q}g(t), \ \ t\in(0,1].
\]
\end{lem}
\begin{Proof}
We define $g(t)=Kt^{-1+\frac{1}{c}}$, $t\in(0,1]$, where $c=\oo_q\Big(\dfrac{f^q}{h}\Big)^{1/q}$.
Thus $c$ satisfies $(1-q)c^q+qc^{q-1}=\dfrac{f^q}{h}$.

Let $K$ be such that
\[
\int^1_0g=f\Leftrightarrow K\int^1_0t^{-1+\frac{1}{c}}dt=f\Leftrightarrow Kc=f\Leftrightarrow K=\frac{f}{c}.
\]
For this $K$ we claim that $\intl^1_0g^q=h$. Indeed:
\[
\int^1_0g^q=K^q\int^1_0t^{-q+\frac{q}{c}}dt=\frac{f^q}{c^q}\frac{1}{\Big(1-q
+\dfrac{q}{c}\Big)}
=\frac{f^q}{(1-q)c^q+qc^{q-1}}=\frac{f^q}{f^q/h}=h,
\]
and Lemma \ref{lem4.1} is proved.  \hs
\end{Proof}

From all the above we conclude Theorem 1.
\section{Characterization of the extremal sequences}\label{sec5}
\noindent
{\bf Proof of Theorem 2.} We consider $\phi_n:(X,\mi)\ra\R^+$ such that the hypotheses of Theorem 2 are satisfied.\ That is
\[
\int_X\phi_nd\mi=f,\ \ \int_X\phi^q_nd\mi=h \ \ \text{and} \ \ \lim_n
\int_X(\cm_\ct\phi_n)^qd\mi=h\oo_q(f^q/h).
\]
We will prove that $\dis\lim_n\int\limits_X|\cm_\ct\phi_n-c\phi_n|^qd\mi=0$, where $c=\oo_q(f^q/h)^{1/q}$.

By setting $\De_n=\{\cm_\ct\phi_n\ge c\phi_n\}$ and $\De'_n=X\sm \De_n$, it is enough to prove that if $I_n$ and $J_n$ are defined as
\[
I_n=\int_{\De_n}(\cm_\ct\phi_n-c\phi_n)^qd\mi \ \ \text{and} \ \ J_n=\int_{\De'_n}
(c\phi_n-\cm_\ct\phi_n)^qd\mi
\]
then $I_n,J_n\ra0$, as $n\ra\infty$.

Define the following functions on $(X,\mi)$
\[
g_n=\phi^q_n(\cm_\ct\phi_n)^{q(q-1)} \ \ \text{and} \ \ h_n=(\cm_\ct\phi_n)^{q(1-q)}.
\]
Remember that in the proof of Theorem 1 in Section \ref{sec3} it is used the inequality:
\setcounter{equation}{0}
\begin{eqnarray}
\hspace*{1cm}\int_X\phi^qd\mi\le\bigg[\int_X\phi(\cm_\ct\phi)^{q-1}d\mi\bigg]^q\cdot\bigg[
\int_X(\cm_\ct\phi)^qd\mi\bigg]^{1-q},  \label{eq5.1}
\end{eqnarray}
for every suitable $\phi$.

Thus since $(\phi_n)$ is extremal for (\ref{eq1.7}) we must have equality in (\ref{eq5.1}) in the limit if $\phi$ is replaced by $\phi_n$.\ We can write:
\begin{eqnarray}
\hspace*{1cm}\int_X g_n\cdot h_nd\mi\approx\bigg[\int_Xg^{1/q}_nd\mi\bigg]\cdot\bigg[
\int_Xh^{1/1-q}_nd\mi\bigg]^{1-q}.  \label{eq5.2}
\end{eqnarray}
We need now two lemmas before we proceed to the proof of Theorem 2. The first one is the following
\setcounter{lem}{0}
\begin{lem}\label{lem5.1}
Under the above notation and hypotheses we have that:
\begin{eqnarray}
\hspace*{1cm}\int_{X_n}g_nh_hd\mi\approx\bigg[\int_{X_n}g^{1/q}_nd\mi\bigg]^q\cdot\bigg[
\int_{X_n}h^{1/1-q}_nd\mi\bigg]^{1-q},  \label{eq5.3}
\end{eqnarray}
where $X_n$ may be replaced either by $\De_n$ or $\De'_n$.
\end{lem}
\begin{Proof}
Of course the following inequalities hold true, in view of Holder's inequality.\ These are:
\begin{eqnarray}
\hspace*{1cm}\int_{\De_n}g_nh_nd\mi\le\bigg[\int_{\De_n}g^{1/q}_nd\mi\bigg]^q\cdot\bigg[
\int_{\De_n}h^{1/1-q}_nd\mi\bigg]^{1-q}, \ \ \text{and} \label{eq5.4}
\end{eqnarray}
\begin{eqnarray}
\hspace*{1cm}\int_{\De'_n}g_nh_nd\mi\le\bigg[\int_{\De'_n}g^{1/q}_nd\mi\bigg]^q\cdot
\bigg[\int_{\De'_n}h^{1/(1-q)}_nd\mi\bigg]^{1-q}. \label{eq5.5}
\end{eqnarray}
We add then and we obtain
\begin{align}
\int_X g_nh_nd\mi\le&\,\bigg[\int_{\De_n}g^{1/q}_nd\mi\bigg]^q\cdot[\int_{\De_n}
h^{1/1-q}_nd\mi\bigg]^{1-q}\nonumber\\
&+\bigg[\int_{\De'_n}g^{1/q}_nd\mi\bigg]^q\cdot
\bigg[\int_{\De'_n}h^{1/1-q}_nd\mi\bigg]^{1-q}.  \label{eq5.6}
\end{align}
We use now the following elementary inequality which proof is given below:\vspace*{0.2cm}

For every $t,t'\ge0$, $s,s'\ge0$ such that $t+t'=a>0$ and $s+s'=b>0$ and any $q\in(0,1)$, we have that
\begin{eqnarray}
\hspace*{1cm}t^q\cdot s^{1-q}+(t')^q\cdot(s')^{1-q}\le a^qb^{1-q}.  \label{eq5.7}
\end{eqnarray}
Applying it on (\ref{eq5.6}) we obtain
\[
\hspace*{0.5cm}\int_Xg_nh_nd\mi\le\bigg[\int_Xg^{1/q}_nd\mi\bigg]^q\cdot\bigg[
\int_Xh^{1/(1-q)}_nd\mi\bigg]^{1-q}
\]
which is equality in the limit.\ As a consequence we must have equality in the limit on (\ref{eq5.4}) and (\ref{eq5.5}) and Lemma \ref{lem5.1} follows.\ It remains to prove the inequality (\ref{eq5.7}).

Fix $t$ such that $0<t<a$ and consider the function $F$ of the variable $s\in[0,b]$ defined by
\[
F(s)=t^q\cdot s^{1-q}+(a-t)^q(b-s)^{1-q}.
\]
It can be easily seen that $F$ is strictly increasing on $\Big[0,t\dfrac{b}{a}\Big]$ and strictly decreasing on $\Big[t\dfrac{b}{a},b\Big]$.\ Thus it attains it's maximum value on $t\dfrac{b}{a}$.\ This maximum value equals to $F\Big(t\dfrac{b}{a}\Big)=a^qb^{1-q}$, thus our inequality is proved. \hs
\end{Proof}

We state now the following:
\begin{lem}\label{lem5.2}
We suppose we are given $w_n:X_n\ra\R^+$ where $X_n\subseteq X$ for any $n\in\N$ such that $w_n\ge w$ on $X_n$ where $w$ is defined on $X$ with non-negative values.\ Suppose also that $q\in(0,1)$ and $\dis\lim_n\int\limits_{X_n}w^q_nd\mi=\dis\lim_n\int\limits_{X_n}w^qd\mi$.\ Then the following is true:
\[
\lim_n\int_{X_n}(w_n-w)^qd\mi=0.
\]
\end{lem}
\begin{Proof}
We set $z_n=w^q_n$ and $z=w^q$ defined on $X_n$ and $X$ respectively.\ We use now the inequality:
\[
x^p-y^p\le px^{p-1}(x-y), \ \ \text{for} \ \ x>y>0, \ \ \text{and} \ \ p>1
\]
which can be proved easily by the mean value theorem on derivatives.

We apply it in case where $p=1/q$.\ Thus we have that
\begin{align*}
w_n-w&=z^p_n-z^p\le pz^{p-1}_n(z_n-z) \\
&=\frac{1}{q}z_n^{\frac{1}{q}-1}(z_n-z), \ \ \text{on} \ \ X_n.
\end{align*}
This gives us
\begin{align*}
\int_{X_n}(w_n-w)^qd\mi&\le\bigg(\frac{1}{q}\bigg)^q\int_Xz^{1-q}_n(z_n-z)^qd\mi \\
&\le\bigg(\frac{1}{q}\bigg)^q\bigg[\int_{X_n}(z_n-z)d\mi\bigg]^q\cdot\bigg[
\int_{X_n}z_n\bigg]^{1-q},
\end{align*}
which is obviously tending to 0 by the hypotheses of the Lemma.\ Note that in the last inequality we use Holder's inequality with exponents $p=1/q$ and $p'=\dfrac{1}{1-q}$.\ Lemma \ref{lem5.1} is now proved. \hs
\end{Proof}

We are now able to continue with the proof of Theorem 2.\vspace*{0.2cm}

We set $\la=\dis\lim_n\Big(\dfrac{h}{\int\limits_X g_n^{1/q}d\mi}\Big)^{1/(1-q)}$ or equivalently:
\[
\la^{1-q}=\lim_n\frac{h}{\int\limits_Xg^{1/q}_nd\mi}.
\]
In view of the equality (\ref{eq5.2}) we must have that
\begin{align*}
\lim_n\int_Xg^{1/q}_nd\mi&=\lim_n\int_X\phi_n(\cm_\ct\phi_n)^{q-1}d\mi \\
&=\frac{h^{1/q}}{\dis\lim_n\Big[\int\limits_X(\cm_\ct\phi_n)^qd\mi\Big]^{(1-q)/q}}=
\frac{h^{1/q}}{\Big[h\oo_q(f^q/h)\Big]^{1/q-1}} \\
&=h\oo_q(f^q/h)^{1-\frac{1}{q}}.
\end{align*}
Thus $\la=\oo_q(f^q/h)^{1/q}=c$.

We remind that $I_n=\int\limits_{\De_n}(\cm_\ct\phi_n-c\phi_n)^qd\mi$, where $\De_n=\{\cm_\ct\phi_n\ge c\phi_n\}$.\ Because now of Lemma 5.2 we have that $I_n\ra0$ if we are able to show that
\begin{eqnarray}
\lim_n\int_{\De_n}(\cm_\ct\phi_n)^qd\mi=c^q\lim_n\int_{\De_n}\phi^q_nd\mi. \label{eq5.8}
\end{eqnarray}
We suppose (by passing to a subsequence if necessary) that
\begin{eqnarray}
\de=\lim\mi(\De_n)\in(0,1).  \label{eq5.9}
\end{eqnarray}
We will discuss the alternative case $\de=0$ or 1 at the end of this section.

We set now
\[
\la_n=\frac{\int\limits_{\De_n}\phi^q_nd\mi}{\int\limits_{\De_n}(\cm_\ct\phi_n)^qd\mi}
\ \ \text{and} \ \ \mi_n=\frac{\int\limits_{\De'_n}\phi^q_nd\mi}{\int\limits_{\De'_n}(\cm_\ct\phi_n)^qd\mi}.
\]
In view of (\ref{eq5.9}) $\la_n,\mi_n$ are well defined for all large n since $\cm_\ct\phi_n\ge f>0$ on $X$.

We set $\la_n=\dfrac{a_n}{b_n}$ and $\mi_n=\dfrac{c_n}{d_n}$ with the obvious meaning on these parameters and suppose without loss of generality that $a_n\ra a_1$, $b_n\ra b_1$, $c_n\ra c_1$ and $d_n\ra d_1$.
Then according to (\ref{eq5.9}) we have that $c_1, d_1>0$.

Because of the definition of $\De_n$ and $\De'_n$ we see immediately that
\begin{eqnarray}
\la_n\le\frac{1}{c^q}\le\mi_n,  \label{eq5.10}
\end{eqnarray}

In order to prove (\ref{eq5.8}) and the respective equality in the case of $J_n$ we need to prove that $\la_n\ra1/c^q$ and $\mi_n\ra1/c^q$.\ So we just need to prove that $\mi_n-\la_n\ra0$. We proceed to this proof as follows

By Section \ref{sec3} we see after replacing $\phi$ by $\phi_n$ that:
\begin{align}
I&=\int_X(\cm_\ct\phi_n)^qd\mi\le\frac{1}{1-q}f^q-\frac{q}{1-q}\int_X\phi_n(\cm_\ct\phi_n)^{q-1}
d\mi \nonumber \\
&=\frac{1}{1-q}f^q-\frac{q}{1-q}\bigg[\int_{\De_n}+\int_{\De'_n}\phi_n
(\cm_\ct\phi_n)^{q-1}d\mi\bigg].  \label{eq5.11}
\end{align}
By using Lemma \ref{lem5.1} and since $(\phi_n)_n$ is extremal for (\ref{eq1.7}), we conclude that
\begin{align}
h\oo_q(f^q/h)\le\frac{f^q}{1-q}-\frac{q}{1-q}\lim_n&\left[\frac{\Big(\int\limits_{\De_n}
\phi^q_nd\mi\Big)^{1/q}}{\Big(\int\limits_{\De_n}(\cm_\ct\phi_n)^qd\mi\Big)^{1/q-1}}\right.\nonumber\\
&\ \ \left.+
\frac{\Big(\int\limits_{\De'_n}\phi^q_nd\mi\Big)^{1/q}}{\Big(\int\limits_{\De'_n}
(\cm_\ct\phi_n)^qd\mi\Big)^{1/q-1}}\right] \label{eq5.12}
\end{align}
We use now Holder's inequality in it's primitive form
\[
\frac{(x+y)^p}{(s+t)^{p-1}}\le\frac{x^p}{s^{p-1}}+\frac{y^p}{t^{p-1}},
\]
for $x,y\ge0$ and $s ,t>0$, $p>1$, which is equality if and only if $\dfrac{x}{s}=\dfrac{y}{t}=k\in\R^+$.

We thus have for $p=1/q>1$, that the expression in brackets in (5.12) is not less than $\dfrac{h^{1/q}}{\Big(\int\limits_X(\cm_\ct\phi_n)^qd\mi\Big)^{1/q-1}}$ which tends to $h\oo_q(f^q/h)^{1-\frac{1}{q}}$.\ So from (5.12) we obtain that
\begin{align}
&h\oo_q(f^q/h)\le\frac{f^q}{1-q}-\frac{q}{1-q}h\oo_q(f^q/h)^{1-\frac{1}{q}}\Leftrightarrow \nonumber\\
&q\oo_q(f^q/h)^{1-\frac{1}{q}}+(1-q)\oo_q(f^q/h)\le f^q/h.  \label{eq5.13}
\end{align}
But by the definition of $\oo_q(z)$, $z\ge1$ we have that (5.13) is equality.\ As a consequence of all the above we conclude that $\dfrac{a_1}{b_1}=\dfrac{c_1}{d_1}\in\R^+$, that is what exactly we wanted to show.

The case $\mi(\De_n)\ra0$ is treated in a similar but more simple way since then
\begin{eqnarray}
\lim_n\int_{\De_n}(\cm_\ct\phi_n)^qd\mi=0.  \label{eq5.14}
\end{eqnarray}
This is true since if we define
\begin{align*}
B_q(f,h,k)=\sup&\bigg\{\int_K(\cm_\ct\phi)^qd\mi:\;\phi\ge0,\int_X\phi d\mi=f,\;\int_X\phi^qd\mi=h, \\
&\ \ K:\mi-\text{measurable with}\ \ \mi(K)=k\bigg\}
\end{align*}
for $0<h\le f^q$ and $k\in(0,1]$, we easily see by it's evaluation in \cite{4} (which is based only on the evaluation of (\ref{eq1.7}) and calculus arguments) that
\[
\lim_{k\ra0^+}B_q(f,h,k)=0
\]
for any fixed $f,h$ such that $0<h\le f^q$.

Thus we end the one direction of Theorem 2.\ For the other one we argue as follows

Since ii) holds we must have that:
\begin{eqnarray}
\hspace*{1cm}\lim_n\int_{\De_n}(\cm_\ct\phi_n-c\phi_n)^qd\mi=0 \ \ \text{and} \ \
\lim_n\int_{\De'_n}(c\phi_n-\cm_\ct\phi_n)^qd\mi=0,  \label{eq5.15}
\end{eqnarray}
with $\De_n$ and $\De'_n$ defined as above.\  We use now the elementary inequality:
\[
0<x^q-y^q\le(x-y)^q \ \ \text{for any} \ \ x>y>0 \ \ \text{and} \ \ q\in(0,1).
\]
So by (\ref{eq5.15}) we must have that
\[
\lim_n\int_{\De_n}(\cm_\ct\phi_n)^qd\mi=c^q\lim_n\int_{\De_n}\phi^q_nd\mi
\]
by passing if necessary to a subsequence, and analogously for $\De'_n$.
Adding these two equalities we obtain i).
Theorem 2 is now proved. \hs
\section{Further properties of extremal sequences}\label{sec6}
\noindent

In Theorem \ref{thm2.1} we stated an equality which relates the dyadic maximal operator with the Hardy operator in an immediate way.\ This equality involves a free parameter which is the function $G_1$.\ In this section we will prove a part of Theorem \ref{thm2.1} for the case $G_1(x)=x^q$, $q\in(0,1)$ and we will use this proof and the statement of Theorem \ref{thm2.1} to find another characterization of some extremal sequences of certain type for the Bellman function of the dyadic maximal operator in relation with Kolmogorov's inequality.\ We proceed to it as follows.\ We prove the following
\setcounter{lem}{0}
\begin{lem}\label{lem6.1}
For any $g:(0,1]\ra\R^+$ integrable and non-increasing for which the integral on the right hand side of the following inequality is finite, we have that:
\[
\int_X(\cm_\ct\phi)^qd\mi\le\int^1_0\bigg(\frac{1}{t}\int^t_0g(u)du\bigg)^qdt,
\]
for any $\phi:(X,\mi)\ra\R^+$ such that $\phi^\ast=g$.
\end{lem}
\begin{Proof}
We have that
\setcounter{equation}{0}
\begin{align}
I&=\int_X(\cm_\ct\phi)^qd\mi=q\int^{+\infty}_{\la=0}\la^{q-1}\mi
(\{\cm_\ct\phi\ge\la\})d\la \nonumber \\
&=q\int^f_{\la=0}\la^{q-1}d\la+q\int^{+\infty}_{\la=f}
\la^{q-1}\mi(\{\cm_\ct\phi\ge\la\})d\la,  \label{eq6.1}
\end{align}
where $f=\int^1_0g=\int\limits_X\phi d\mi$, for any $\phi$ such that $\phi^\ast=g$.

Thus $I=f^q+q\int\limits^{+\infty}_{\la=f}\la^{q-1}\bi(\la)d\la$, where $\bi(\la)=\mi(\{\cm_\ct\phi\ge\la\})$.

By inequality (1.2) we see that
\[
\mi(\{\cm_\ct\phi\ge\la\})=\bi(\la)\le\frac{1}{\la}\int_{\{\cm_\ct\phi\ge\la\}}\phi
d\mi, \ \ \text{for any} \ \ \la>f.
\]
Thus
\[
\frac{1}{\bi(\la)}\int_{\{\cm_\ct\phi\ge\la\}}\phi d\mi\ge\la.
\]
Since $\phi^\ast=g$ is non-increasing we have that
\begin{eqnarray}
\int_{\{\cm_\ct\phi\ge\la\}}\phi d\mi\le\int^{\bi(\la)}_0g(u)du.  \label{eq6.2}
\end{eqnarray}

We now choose for any $\la>f$ the unique $a(\la)\in(0,1]$ such that $\dfrac{1}{a(\la)}\intl^{a(\la)}_0g=\la$.\ As a consequence we have that
\begin{eqnarray}
\frac{1}{\bi(\la)}\int^{\bi(\la)}_0g\ge\frac{1}{\bi(\la)}\int_{\{\cm_\ct\phi\ge\la\}}\phi
d\mi\ge\la=\frac{1}{a(\la)}\int^{a(\la)}_0g,  \label{eq6.3}
\end{eqnarray}
and since $g:(0,1]\ra\R^+$ is non-increasing we conclude that $\bi(\la)\le a(\la)$.\ Thus from (\ref{eq6.1}) we derive the following
\begin{align}
I&=\int_X(\cm_\ct\phi)^qd\mi\le f^q+q\int^{+\infty}_{\la=f}\la^{q-1}a(\la)d\la \nonumber \\
&=f^q+\int^{+\infty}_{\la=f}q\la^{q-1}\bigg|\bigg\{t\in(0,1]\cdot
\frac{1}{t}\int^t_0g\ge\la\bigg\}\bigg|d\la,  \label{eq6.4}
\end{align}
from the definition of $a(\la)$.\ So (\ref{eq6.4}) gives
\[
I=\int_X(\cm_\ct\phi)^qd\mi\le\int^1_0\bigg(\frac{1}{t}\int^t_0g\bigg)^qdt,
\]
for any $\phi$ such that $\phi^\ast=g$, which is the result that is stated in our Lemma. \hs
\end{Proof}

We will also need the following.
\begin{lem}\label{lem6.2}
Let $(\phi_n)_n$ be such that $\phi_n:(X,\mi)\ra\R^+$ are measurable rearrangements of $g$ $(\phi^\ast_n=g)$, such that
\begin{eqnarray}
\lim_n\int_X(\cm_\ct\phi_n)^qd\mi=\int^1_0\bigg(\frac{1}{t}\int^t_0g\bigg)^qdt,
\label{eq6.5}
\end{eqnarray}
Then the following is true.\ For any $k\in(0,1]$
\[
\lim_n\int^k_0[(\cm_\ct\phi_n)^\ast(t)]^qdt=\int^k_0\bigg(\frac{1}{t}\int^t_0g\bigg)^qdt.
\]
\end{lem}
\begin{proof}
We suppose that (\ref{eq6.5}) is true.\ Then in view of the proof of Lemma \ref{lem6.1} we must have that
\begin{eqnarray}
\lim_n\int^{+\infty}_{\la=f}q\la^{q-1}\mi(\{\cm_\ct\phi_n\ge\la\})d\la=
\int^{+\infty}_{\la=f}q\la^{q-1}[a(\la)]d\la, \label{eq6.6}
\end{eqnarray}
where $f=\int^1_0g=\int\limits_X\phi_nd\mi$ and $a(\la)$ is as in the proof of Lemma \ref{lem6.1}.

This means that the following equality should be true
\begin{eqnarray}
\lim_n\int^{+\infty}_{\la=f}q\la^{q-1}|\{(\cm_\ct\phi_n)^\ast\ge\la\}|d\la=\int^{+\infty}_{\la=f}
q\la^{q-1}[a(\la)]d\la,  \label{eq6.7}
\end{eqnarray}
since $\mi(\{\cm_\ct\phi\ge\la\})=|\{(\cm_\ct\phi)^\ast\ge\la\}|$, for any $\la>0$, where $|\cdot|$ denotes the Lesbesgue measure on $(0,1]$.\ Then for any $k\in(0,1]$ we have that
\begin{align}
I_n&=\int^k_0[(\cm_\ct\phi_n)^\ast(t)]^qdt=\int^{+\infty}_{\la=0}q\la^{q-1}
|\{t\in(0,k]:(\cm_\ct\phi_n)^\ast(t)\ge\la\}|d\la \nonumber\\
&=kf^q+\int^{+\infty}_{\la=f}q\la^{q-1}|\{t\in(0,k]:(\cm_\ct\phi_n)^\ast(t)\ge\la\}|d\la.
\label{eq6.8}
\end{align}
We set now $(\cm_\ct\phi_n)^\ast(k)=\la^{(n)}_k\in[f,+\infty)$, for any $n\in\N$.\ For any $t\in(0,k]:(\cm_\ct\phi_n)^\ast(t)\ge(\cm_\ct\phi_n)^\ast(k)=\la^{(n)}_k\Rightarrow
\forall\;\la\in[f,\la^{(n)}_k]$, so we must have that
\begin{eqnarray}
|\{t\in(0,k]:(\cm_\ct\phi_n)^\ast(t)\ge\la\}|=|(0,k]|=k.  \label{eq6.9}
\end{eqnarray}
By (\ref{eq6.8}) and (\ref{eq6.9}) we have that
\begin{align}
I_n=&\,\int^k_0[(\cm_\ct\phi_n)^\ast]^qdt=kf^q+\int^{\la^{(n)}_k}_{\la=f}
q\la^{q-1}\cdot kd\la \nonumber \\
&+\int^{+\infty}_{\la=\la^{(n)}_k}q\la^{q-1}|\{t\in(0,k]:(\cm_\ct\phi_n)^\ast(t)\ge
\la\}|d\la \nonumber\\
=&\,
\,k(\la^{(n)}_k)^q+\int^{+\infty}_{\la=\la^{(n)}_k}q\la^{q-1}|\{t\in(0,1]:(\cm_\ct\phi_n)^\ast
(t)\ge\la\}|d\la,  \label{eq6.10}
\end{align}
by the definition of $\la^{(n)}_k$.\ Thus
\begin{align}
I_n=&\,k(\la^{(n)}_k)^q+\int^{+\infty}_{\la=f}q\la^{q-1}|\{t\in(0,1]:(\cm_\ct\phi_n)^\ast(t)\ge\la\}|d\la \nonumber\\
&-\int^{\la^{(n)}_k}_{\la=f}q\la^{q-1}|\{t\in(0,1]:(\cm_\ct\phi_n)^\ast(t)\ge\la\}|d\la \nonumber \\
=&\,k(\la^{(n)}_k)^q+I_1-I_2, \ \ \text{say}. \label{eq6.11}
\end{align}
Concerning $I_1$ we have that
\begin{eqnarray}
I_1\ra\int^{+\infty}_{\la=f}q\la^{q-1}[a(\la)]d\la, \label{eq6.12}
\end{eqnarray}
as $n\ra\infty$ by the comments in the beginning of the proof of this Lemma.\ About $I_2$ we have that
\begin{align}
I_2&=\int^{\la^{(n)}_k}_{\la=f}q\la^{q-1}|\{t\in(0,1]:(\cm_\ct\phi_n)^\ast(t)\ge\la\}|d\la \nonumber\\
&=\int^{\la^{(n)}_k}_{\la=f}q\la^{q-1}\bi_n(\la),  \ \ \text{where} \label{eq6.13}
\end{align}
\[
\bi_n(\la)=\mi(\{\cm_\ct\phi_n\ge\la\})\le a(\la),
\]
since $\phi^\ast_n=g$ and the proof of Lemma \ref{lem6.1}.\ as a consequence
\begin{eqnarray}
I_2\le\int^{\la^{(n)}_k}_{\la=f}q\la^{q-1}a(\la)d\la.  \label{6.14}
\end{eqnarray}
Thus
\begin{eqnarray}
\underset{n}{\lim\inf}J_n\ge\int^{+\infty}_{\la=f}q\la^{q-1}[a(\la)]d\la+
\underset{n}{\lim\inf}\bigg[k(\la^{(n)}_k)^q-\int^{\la^{(n)}_k}_{\la=f}q
\la^{q-1}a(\la)d\la\bigg].\hspace*{-2cm}  \label{eq6.15}
\end{eqnarray}
Since now $k$ is fixed and positive and because of the fact that $\dis\sup_n\int\limits_X(\cm_\ct\phi_n)^qd\mi<+\infty$ (which can be proved by using Lemma \ref{lem6.1}), we conclude that $(\la^{(n)}_k)_n$ is bounded above.\ Thus there exists a subsequence and a $\la_0\ge f$ such that $\la^{(n_i)}_k\ra\la_0$, as $i\ra\infty$.\ Thus
\begin{align}
\underset{i}{\lim\inf}J_{n_i}\ge&\,\underset{i}{\lim\inf}J_n\ge
\int^{+\infty}_{\la=f}q\la^{q-1}[a(\la)]d\la+k\la^q_0
-\int^{\la_0}_{\la=f}q\la^{q-1}[a(\la)]d\la \nonumber \\
=&\,k\la^q_0+\int^{+\infty}_{\la=\la_0}q\la^{q-1}[a(\la)]d\la \nonumber\\
=&\,k\la^q_0+\int^{+\infty}_{\la=\la_0}q\la^{q-1}\bigg|\bigg\{t\in(0,1]:\frac{1}
{t}\int^t_0g\ge\la\bigg\}\bigg|d\la \nonumber \\
\ge&\,\int^{\la_0}_{\la=0}q\la^{q-1}kd\la+\int^{+\infty}_{\la=\la_0}\bigg|\bigg\{
t\in(0,k]:\frac{1}{t}\int^t_0g\ge\la\bigg\}\bigg|d\la \nonumber\\
\ge&\,\int^{+\infty}_{\la=0}q\la^{q-1}\bigg|\bigg\{t\in(0,k]:\frac{1}{t}\int^t_0g\ge\la
\bigg\}\bigg|d\la
=\int^k_0\bigg(\frac{1}{t}\int^t_0g\bigg)^qdt.  \label{eq6.16}
\end{align}
That is we proved that for any fixed $k\in(0,1]$ there is a subsequence of integers $(n_j)_j$ such that
\[
\lim_n\int^k_0[(\cm_\ct\phi_{n_j})^\ast]^qdt\ge\int^k_0\bigg(\frac{1}{t}\int^t_0
g\bigg)^qdt.
\]
This consequence, Lemma \ref{lem6.1}, and standard arguments about subsequences give the result we need.
\end{proof}

We are now able to prove the main theorem of this section.
\begin{thm}\label{thm6.1}
Let $g:(0,1]\ra\R^+$ be an integrable, non-increasing function such that $\int^1_0\Big(\dfrac{1}{t}\int^t_0g\Big)^qdt<+\infty$ where $q\in[0,1)$ and $(\phi_n)$ is a sequence of $\mi$-measurable rearrangements of $g$ $(\phi^\ast_n=g)$ such that
\[
\lim_n\int_X(\cm_\ct\phi_n)^qd\mi=\int^1_0\bigg(\frac{1}{t}\int^t_0g\bigg)^qdt.
\]
Then the following equality is true
\[
\lim_n\int^1_0\bigg|(\cm_\ct\phi_n)^\ast(t)-\frac{1}{t}\int^t_0g\bigg|^qdt=0.
\]
\end{thm}
\begin{proof}
We consider the set
\[
F_n=\bigg\{t\in(0,1]:(\cm_\ct\phi_n)^\ast(t)>\frac{1}{t}\int^t_0g\bigg\}
\]
and it's complement in $(0,1]$, $F^c_n$.

We will prove that
\begin{eqnarray}
\lim_n\bigg|\int_{X_n}[(\cm_\ct\phi_n)^\ast(t)]^qdt-\int_{X_n}\bigg(\frac{1}{t}
\int^t_0g\bigg)^qdt\bigg|=0,  \label{eq6.17}
\end{eqnarray}
where $X_n$ is either $F_n$, for every $n\in\N$, or $F^c_n$, $\forall\;n\in\N$.\ If we have (\ref{eq6.17}) in both cases for $X_n$  and apply Lemma \ref{lem5.2}, then we have the result we need to prove.\ We will prove (\ref{eq6.7}) only in the case where $X_n=F_n$, $\forall\;n\in\N$.\ The other one is treated in a similar way.

For every $n\in\N$ we choose $U_n$ to be an open subset of $(0,1]$ such that $F_n\subseteq U_n$ and $|U_n\setminus F_n|\le\dfrac{1}{n}$.\ Then $U_n$ can be written as $U_n=\bigcup_k(a^{(n)}_k,b^{(n)}_k)$, that is a disjoint union of open intervals on $(0,1]$.\ By Lemma \ref{lem6.2} and since the above union is disjoint we have that
\begin{eqnarray}
\lim_n\bigg|\int_{U_n}[(\cm_\ct\phi_n)^\ast(t)]^qdt-\int_{U_n}\bigg(
\frac{1}{t}\int^t_0g\bigg)^qdt\bigg|=0.  \label{eq6.18}
\end{eqnarray}
Moreover
\[
\int_{U_n\setminus F_n}[(\cm_\ct\phi_n)^\ast(t)]^qdt\le\int^{|U_n\setminus F_n|}_0
[(\cm_\ct\phi_n)^\ast(t)]^qdt,
\]
for any $n\in\N$, since $(\cm_\ct\phi_n)^\ast$ is non-increasing on $(0,1]$.

Additionally by Theorem \ref{thm2.1} it is easy to see that
\[
\int^{|U_n\setminus F_n|}_0[(\cm_\ct\phi_n)^\ast(t)]^qdt\le\int^{|U_n\setminus F_n|}_0
[H(t)]^qdt, \ \ \text{where} \ \ H(t)=\frac{1}{t}\int^t_0g.
\]
The last integral now tends to zero since by our hypothesis $H(t)\in L^q((0,1])$.\ For the same reasons we have that
\[
\int_{U_n\setminus F_n}\bigg(\frac{1}{t}\int^t_0g\bigg)^qdt\ra0, \ \ \text{as} \ \ n\ra\infty
\]
for the same reasons.\ Combining the above results we conclude, by using also (\ref{eq6.18}), the proof of Theorem \ref{thm6.1}.
\end{proof}
\begin{rem}\label{rem6.1}
By using the elementary inequality $x^q-y^q<(x-y)^q$, for $x>y>0$, it is easy to see that the converse statement of Theorem \ref{thm6.1} is true.\ That is any sequence satisfying
\[
\lim_n\int^1_0\bigg|(\cm_\ct\phi_n)^\ast(t)-\frac{1}{t}\int^t_0g\bigg|^qdt=0,
\]
must also satisfy:
\[
\lim_n\int_X(\cm_\ct\phi_n)^qd\mi=\int^1_0\bigg(\frac{1}{t}\int^t_0g\bigg)^qdt.
\]
\end{rem}
\begin{cor}\label{cor6.1}
Let $g$ be as in Section \ref{sec4}, Lemma \ref{lem6.1}.\ Then for any sequence $(\phi_n)$ of rearrangements of $g$ such that $\int\limits_X(\cm_\ct\phi_n)^qd\mi\ra h\oo_q(f^q/h)$, we must have that
\[
\int^1_0|(\cm_\ct\phi_n)^\ast(t)-cg(t)|^qdt\ra0, \ \ \text{as} \ \ n\ra\infty \ \ \text{where} \ \ c=\oo_q(f^q/h)^{1/q}.
\]
\end{cor}
\begin{proof}
Immediate.
\end{proof}

Nikolidakis Eleftherios, Post-doctoral researcher, National and Kapodistrian University of Athens, Department of Mathematics, Panepistimioupolis GR-157 84, Athens, Greece.

E-mail address:lefteris@math.uoc.gr

\end{document}